\documentclass[12pt,a4paper]{amsart}
\usepackage[T2A]{fontenc}
\usepackage[cp1251]{inputenc}

\usepackage{graphics}
\usepackage{epsfig}
\usepackage{graphicx}

\usepackage{amssymb}

\advance\hoffset-20mm \advance\textwidth40mm

\newtheorem{theorem}{Theorem}[section]
\newtheorem{lemma}[theorem]{Lemma}
\newtheorem{theo}[theorem]{Theorem}
\newtheorem*{maintheo}{Main Theorem}
\newtheorem*{llemma}{Lemma}
\newtheorem{proposition}[theorem]{Proposition}
\newtheorem{corollary}[theorem]{Corollary}
\theoremstyle{definition}
\newtheorem*{definition}{Definition}

\theoremstyle{remark}
\newtheorem*{remark}{Remark}

\makeatletter
\renewcommand{\subsubsection}{\@startsection{subsubsection}{3}{0pt}{-3.25ex plus -1ex minus-.2ex}{1.5ex plus.2ex}{\normalfont\bfseries}}
\makeatother

\def\diag{\mathop{\rm diag}}
\def\Gr{\mathop{\rm Gr}}

\def\kk{{\Bbbk}}

\def\ZZ{{\mathbb Z}}
\def\RR{{\mathbb R}}

\def\NN{{\mathbb N}}
\def\QQ{{\mathbb Q}}

\def\conv{{\rm conv}}

\def\midd{{\rm mid}}

\makeatletter
\def\kratno{\mathbin{\lower.2ex\hbox{$\m@th\vdots$}}}
\def\notkratno{\mathbin{\lower.2ex\hbox{$\m@th\not\vdots$}}}
\makeatother

\newcounter{itemnumber}

\begin{document}

\sloppy
\title[Normal Closures of Maximal Torus Orbits]{Simple
$SL(n)$-Modules\\with Normal Closures of Maximal Torus Orbits}
\author%
{K. Kuyumzhiyan}
\address{Department of Higher Algebra, Faculty of Mechanics and Mathematics,
Moscow State University, 119992 Moscow, Russia }
\email{karina@mccme.ru}

%\mathclass2000 \primary{14R20} \secondary{14M25,05E15}%
%
\begin{abstract}
Let $T$ be the subgroup of diagonal matrices in the group $SL(n)$.
The aim of this paper is to find all finite-dimensional simple
rational $SL(n)$-modules $V$ with the following property: for each
point $v\in V$ the closure $\overline{Tv}$ of its $T$-orbit is a
normal affine variety. Moreover, for any $SL(n)$-module without this
property a $T$-orbit with non-normal closure is constructed. The
proof is purely combinatorial: it deals with the set of weights of
simple $SL(n)$-modules. The saturation property is checked for each
subset in the set of weights.
\end{abstract}

 \maketitle

%%%%%%%%%%%%%%%%%%%%%%%%%%%%%%%%%%%%%%%%%%%%%%%%%%%%%%%%%%%%%%%%%%%%%%%%
\section*{Introduction}

Let $T$ be an algebraic torus defined over an algebraically closed
field $\kk$ of characteristic zero. Recall that an irreducible
algebraic $T$-variety $X$ is called \textit{toric} if $X$ is normal and $T$
acts on $X$ with an open orbit. This class of varieties plays an
important role in algebraic geometry, topology and combinatorics due
to its remarkable description in terms of convex geometry,
see~\cite{Ful}. Assume that the torus $T$ acts on a variety $Y$.
Then the closure $X=\overline{Ty}$ of the $T$-orbit of a point $y\in
Y$ is a natural candidate to be a toric variety. To verify it, one
should check that $X$ is normal.

 During last decades, normality of
torus orbits' closures was an object of numerous investigations.
 For example, let $G$ be a semisimple algebraic group with a Borel
subgroup $B$ and a maximal torus $T\subset B$.
 In~\cite{Kl}, it was proved that the closure of a
general $T$-orbit on the flag variety $G/B$ is normal. Later it was
shown that the closure of a general $T$-orbit in $G/P$, where
$P\subset G$ is a parabolic subgroup, is also normal, see~\cite{Da}.
Examples of non-normal closures of non-general torus orbits can be
found in~\cite{GP}.

Now let us consider a finite-dimensional rational $T$-module $V$.
There exists an easy combinatorial criterion of normality of
$\overline{Tv}$ for a vector $v\in V$. Namely, let
$v=v_{\chi_1}+\dots+v_{\chi_m}$, $v_{\chi_i}\neq 0$, be the weight
decomposition of the vector $v$. Consider the corresponding set of
T-weights $\chi_1, \ldots,\chi_m$. If we take $\chi_1,
\ldots,\chi_m$ as elements of the character lattice $\mathfrak X(T
)$, we can generate a semigroup $\ZZ_+(\chi_1, \ldots,\chi_m)$, a
sublattice $\ZZ(\chi_1, \ldots,\chi_m)$, and a rational polyhedral
cone $\QQ_+(\chi_1, \ldots,\chi_m)$. The set $\chi_1, \ldots,\chi_m$
is called \textit{saturated} if $\ZZ_+(\chi_1,
\ldots,\chi_m)=\ZZ(\chi_1, \ldots,\chi_m)\cap\QQ_+(\chi_1,
\ldots,\chi_m)$. It is well known (see~\cite[page 5]{KKMSD}) that
the following two conditions are equivalent: the set $\{\chi_1,
\ldots,\chi_m\}$ is saturated and the closure $\overline{Tv}$ of the
$T$-orbit $Tv$ is normal. There is an analogous criterion for the
$T$-action on the projectivisation $\mathbb P(V)$, see~\cite{GP}.

The saturation property occurs in many algebraic and geometric
problems. In~\cite{Wh}, it was proved that the set of incidence
vectors of the bases of a realizable matroid is saturated. The
geometric conclusion of this fact is that for any point $y$ in the
affine cone over the classical Grassmannian $\Gr(k,n)$ the closure
$\overline{Ty}$ is normal.

Taken a finite graph $\Gamma$ with $n$ vertices, one can associate a
finite collection $M(\Gamma)$ of vectors in the lattice $\ZZ^n$ with
it:
$$
M(\Gamma)=\{\varepsilon_i+\varepsilon_j\mid (ij)\;\text{is an edge
of}\;\Gamma\},
$$
where $\varepsilon_1,\varepsilon_2,\ldots,\varepsilon_n$ is the standard basis of $\ZZ^n$. The saturation property for this set is equivalent
to the fact that for arbitrary two minimal odd cycles $C$ and $C'$
in $\Gamma$, either $C$ and $C'$ have a common vertex or there
exists an edge of $\Gamma$ joining a vertex of $C$ with a vertex of
$C'$ (see~\cite{3} and \cite{5}). Algebraically, the saturation
property for~$M(\Gamma)$ is equivalent to the integral closureness
for the subalgebra $\mathcal A(\Gamma)$ of the polynomial algebra
$\kk[x_1,x_2,\ldots,x_n]$,
$$
\mathcal A(\Gamma)=\kk[x_ix_j\mid (ij)\;\text{is an edge of}\;\Gamma]
$$
in its field of fractions $Q\mathcal A(\Gamma)$.

Some general results concerning quivers and the saturation property
were obtained in~\cite{Chi}. It was shown that a finite, connected
quiver $Q$ without oriented cycles is a Dynkin or Euclidean quiver
if and only if all orbit semigroups of representations of $Q$ are
saturated.

In the paper~\cite{JM}, the following problem is solved. Let $G$ be
a semisimple algebraic group with a maximal torus $T$ and $V$ be its
adjoint module. For which $G$ for all $v\in V$ the closure
$\overline{Tv}$ is normal? The surprising fact is that for $G=SL(n)$
this is always the case (see also~\cite[Ex. 3.7]{Stu}, \cite{Stumf},
\cite{JM}~and~\cite[Prop.2.1]{BZ}). In~\cite{BZ}, this combinatorial
result is interpreted in terms of representations of quivers.

\bigskip

The aim of this paper is to classify all simple finite-dimensional
rational $SL(n)$-modules $V$ such that for any $v\in V$ the closure
$\overline{Tv}$ is normal.

%%%%%%%%%%%%%%%%%%%%%%%%%%%%%%%%%%%%%%%%%%%%%%%%%%%%%%%%%%%%%%%%%%%%%%%

\begin{maintheo}
The representations below, together with their dual, form the list
of all irreducible representations of $SL(n)$ where all maximal
torus orbits' closures are normal:

\begin{enumerate}
\item[(1)] the tautological representation of $SL(n)$;

\item[(2)] the adjoint representation of $SL(n)$;

\item[(3)] exceptional cases:

\smallskip
\begin{center}
\begin{tabular}{|c|c|c|}
\hline Group & Highest weight& $G$-module\\
\hline
$SL(2)$ &$3\pi_1$&$S^3\kk^2$ \\
\hline
$SL(2)$ &$4\pi_1$&$S^4\kk^2$ \\
\hline
$SL(3)$&$2\pi_1$&$S^2\kk^3$\\
\hline
$SL(4)$&$\pi_2$&$\Lambda^2\kk^4$\\
\hline
$SL(5)$&$\pi_2$&$\Lambda^2\kk^5$\\
\hline
$SL(6)$&$\pi_2$&$\Lambda^2\kk^6$\\
\hline
$SL(6)$&$\pi_3$&$\Lambda^3\kk^6$\\
\hline
\end{tabular}
\end{center}
\smallskip

\end{enumerate}

\end{maintheo}

The paper is organized as follows. In Section~1 we give some
algebraic definitions and reformulate the problem in combinatorial
terms. From that point, it remains to check the saturation property
for any subset in the system of $T$-weights of a simple
$SL(n)$-module. In Section~2 we prove that the saturation property
holds for each subset in the set of weights of the representations
listed in the Main Theorem. We also give a new proof of
theorem~\cite[Thm.1]{JM}. It is the most non-trivial positive case,
where the dimension of $V$ is not bounded. When possible, reasoning
uses the graph theory language. Our reference for graph theory
is~\cite{Har}. In Section~3 we produce non-saturated subsets in sets
of weights for all other representations. If the set of weights of
the representation with the highest weight $\lambda$ is a subset in
the set of weights of the representation with the highest weight
$\mu$, and a non-saturated subset for $\lambda$ is known, then one
can use it as a non-saturated subset for $\mu$. Fundamental
representations form the most difficult case. To work with them, we
use the following observation. If a non-saturated subset in the set
of weights of the $k$th fundamental representation of $SL(n)$ is
found, then the analogous non-saturated subset exists in the set of
weights of the $k$th fundamental representation of $SL(n+k)$.

Some auxiliary results used in the proof may be of independent
interest.

\begin{llemma}[Reformulation of Lemma~\ref{utv0}]
Let $M_6=\{(\varepsilon_1,\ldots,\varepsilon_6)\mid \varepsilon_i=\pm 1,
\sum \varepsilon_i=0\}$ be the set of points in $\RR^6$. Then each
subset of this set is saturated.
\end{llemma}

In further publications, we plan to give a classification of simple
$G\text{-modules}$ with normal $T\text{-orbit}$ closures for other
simple algebraic groups $G$.

The author is grateful to her scientific supervisor I.V. Arzhantsev
for the formulation of the problem and fruitful discussions. Thanks
are also due to I.I. Bogdanov for useful comments.

\section{Algebraic background and notation}

Let $V$ be a finite-dimensional rational $T$-module. Given any
character $\chi$ from the character lattice $\mathfrak X(T)$, define
a weight subspace $V_{\chi}$ as $V_{\chi}=\{v\in V\mid t\cdot
v=\chi(t)v\}$. It is well known that
$V=\bigoplus\limits_{\chi\in\mathfrak X(T)}V_{\chi}$, and only
finitely many $V_{\chi}$ are nonzero. The set
$\{\chi_1,\chi_2,\ldots,\chi_k\}$ of those $\chi_i$ for which
$V_{\chi_i}\neq 0$ is called \textit{the system of weights} of the
$T$-module V.

Let $\ZZ_+$ and $\QQ_+$ denote the sets of integer and rational
non-negative numbers, respectively; and let
$v_1,v_2,\ldots,v_m\in\QQ^n$. Consider the semigroup $\ZZ_+(v_1,
v_2, \dots, v_m)=\{n_1v_1+n_2v_2+\ldots+n_mv_m\mid
n_i\in\ZZ_+\}\subseteq\mathfrak X(T)$, the sublattice $\ZZ(v_1, v_2,
\dots, v_m)=\{z_1v_1+z_2v_2+\ldots+z_mv_m\mid
z_i\in\ZZ\}\subseteq\mathfrak X(T)$, and the rational polyhedral
cone $\QQ_+(v_1, v_2, \dots, v_m)=\{q_1v_1+q_2v_2+\ldots+q_mv_m\mid
q_i\in\QQ_+\}\subseteq\mathfrak X(T)\otimes_{\ZZ}\QQ$. Define the
following important property of the set $\{v_1,v_2,\ldots,v_m\}$.

\begin{definition}
The set of points $\{v_1, v_2, \dots, v_m\}\subset\QQ^n$ is called
\emph{saturated} if
$$
\ZZ_+ (v_1, v_2, \dots, v_m) = \ZZ (v_1, v_2, \dots, v_m) \cap \QQ_+
(v_1, v_2, \dots, v_m).
$$
\end{definition}

The following result provides a well-known combinatorial criterion
of normality of the torus orbit closure, see~\cite{KKMSD}:

\begin{theo}
Consider a rational linear action of a torus $T$ on a vector space $V$. Let $v=v_{\lambda_1}+\dots+v_{\lambda_s}$, $v_{\lambda_i}\neq 0$, be
its weight decomposition. Then the
closure $\overline{Tv}$ is normal if and only if the set of
characters $\{\lambda_1, \dots, \lambda_s\}$ is saturated.
\end{theo}

\begin{corollary}
Given a rational linear action of torus $T$ on a vector space $V$;
let $\{\lambda_1, \ldots, \lambda_s\}$ be the set of weights of this
action. Then the closure $\overline{Tv}$ is normal for each $v\in V$
if and only if each subset in $\{\lambda_1, \dots, \lambda_s\}$ is
saturated.
\label{cor12}
\end{corollary}

\begin{remark}
The weight system is multiplied by $-1$ while changing a
representation $V$ of torus $T$ with its dual representation. Hence
the property of normality of all $T$-orbits is preserved.
\end{remark}

Let $G=SL(n)$. We fix a maximal torus $T\subset G$ consisting of all
diagonal matrices. An element $a=(a_1,a_2,\ldots,a_n)$ of the
lattice $\ZZ^n$ can be interpreted as a character $\chi_a$ of the
torus $T$ in the following way: $\chi_a(t)=t_1^{a_1}t_2^{a_2}\ldots
t_n^{a_n}$, where $t=\diag (t_1,t_2,\ldots,t_n)$. Since
$t_1t_2\ldots t_n=1$, the points $a$ and $b$ define the same
character if and only if $a-b=\alpha (1,1,\ldots,1)$. Each $a$ from
$\ZZ^n$ has a unique representation
$a=\tilde{a}+\alpha(1,1,\ldots,1)$, where $\tilde{a}\in\QQ^n$,
$\alpha\in\QQ$, and $\sum\tilde{a_i}=0$.

Let $\varepsilon_1,\varepsilon_2,\ldots,\varepsilon_n$ be the
standard basis of the lattice~$\ZZ^n$, and
$$
e_i=\widetilde{\varepsilon_i}=\left(-\frac 1n,-\frac
1n,\ldots,-\frac 1n,
%\stackrel{i\text{th place}}{\check{\frac{n-1}{n}}},
\mathop{\frac{n-1}{n}}\limits_{\mathstrut i\text{th place}},
-\frac 1n,\ldots,-\frac 1n\right).
$$
Notice that $e_1,e_2,\ldots,e_n$ (further referred to as a {\it
quasi-basis}) satisfy the only linear relation
$$
e_1+e_2+\ldots+e_n=0. \eqno(*)
$$
Identify $\mathfrak X(T)$ with the $\ZZ$-lattice generated by
$e_1,e_2,\ldots, e_n$. Recall that a weight $\chi_a$ is called
\textit{dominant} if and only if $a_1\geqslant a_2\geqslant\ldots\geqslant
a_n$. The \textit{root lattice} for $SL(n)$ is a lattice generated by the vectors $e_1-e_2$, $e_2-e_3,\ldots,e_{n-1}-e_n$. Due to the ambiguity of notation,
$$
\Phi=\{a_1e_1+a_2e_2+\ldots+a_ne_n\mid a_1+a_2+\ldots+a_n\kratno
n\}.
$$

For a positive integer $s\mid n$ (i.e. $n=ss', s'\in\ZZ$), define
$$
\ZZ_{\equiv 0(s)}(e_1,\dots,e_n)=\left\{\sum\nolimits_{i=1}^n
x_ie_i\mid x_i\in\ZZ,\;\sum\nolimits_{i=1}^n x_i\kratno s\right\}
\;.
$$
In this notation the root lattice $\Phi$ coincides with $\ZZ_{\equiv
0(n)}(e_1,\dots,e_n)$.

Let $V$ be a finite-dimensional simple rational $SL(n)$-module,
$M(V)$ be the system of weights of the module~$V$ with respect to
the restricted action $T\colon V$. Introduce a partial order on
$M(V)$: $\mu\succeq\nu$ if and only if for $\xi=\mu-\nu$ the
following conditions hold: $\xi_1\geqslant0$,
$\xi_1+\xi_2\geqslant0$,
$\ldots,\xi_1+\xi_2+\ldots+\xi_{n-1}\geqslant0$. It is well known
that $M(V)$ contains the only maximal element $\lambda$ with respect
to $\succeq$, it is called \textit{the highest weight} of the
module. The weight $\lambda$ is dominant, moreover, for any dominant
weight $\lambda\in\mathfrak X(T)$ there exists a unique simple
$SL(n)$-module $V(\lambda)$ with the highest weight~$\lambda$
(see~\cite[\S 20,\S 21]{Hum}). The role of the Weyl group $W$ is
played here by the permutation group $S_n$, which acts on $\ZZ^n$ by
permutations of coordinates. It is well known (see~\cite[\S 13,\S
21]{Hum} or ~\cite[Chapter 4]{VO}) that
$$
M(\lambda):=M(V(\lambda))=\conv\{w\lambda\mid w\in
W\}\cap(\lambda+\Phi),
$$
where $\conv(M)$ denotes the convex hull of the set $M\subset\RR^n$.
%%%%%%%%%%%%%%%%%%%%%%%%%%%%%%%%%%%%%%%

In our situation, Corollary~\ref{cor12} can be reformulated in the following way:
\begin{proposition}
Let $V(\lambda)$ be a simple module of a semisimple group $G$ with
the highest weight $\lambda$. Then the closure of each $T$-orbit in
$V(\lambda)$ is normal if and only if each subset in $M(\lambda)$ is
saturated.
\end{proposition}

\section{Positive results}

In this section we prove that certain sets of weights are saturated.
We use the following lemmas, their proof can be found in~\cite{JM}.

\begin{lemma}
Let $M$ be a non-saturated set and $\alpha$ be a vector such that
$\alpha\in M$ and $-\alpha\in M$. Then either $M\backslash\{
\alpha\}$ or $M\backslash \{-\alpha\}$ is non-saturated.\label{l.o-}
\end{lemma}

\begin{lemma}
Any set of linearly independent vectors is saturated.\label{l.ln}
\end{lemma}

\begin{lemma}
Let $v=q_1v_1+\dots+q_mv_m$, where $v, v_i$ are arbitrary vectors, and
$q_i\in\mathbb Q_+$. \label{l.oln} Then one can choose a linearly
independent subset
$\{v_{i_1},\dots,v_{i_s}\}\subset\{v_1,\dots,v_m\}$ and numbers
$q_{i_1}',\dots,q_{i_s}'\in\QQ_+$ such that
$$
v=q_{i_1}'v_{i_1}+\dots+q_{i_s}'v_{i_s}.
$$
\end{lemma}

\subsection{The tautological representation.}
\label{tavt} The highest weight of the tautological representation
equals $e_1$. One should prove that each subset in
$\{e_1,\dots,e_n\}$ is saturated. Using Lemma~\ref{l.ln}, we obtain
that any proper subset of this set is saturated. It is easy to see
that this set itself is also saturated because if one takes a
$\ZZ$-combination of vectors $e_i$, then it becomes a
$\ZZ_+$-combination after adding $(*)$ with a positive coefficient.

\subsection{The adjoint representation.}\label{ad}
Its highest weight $\lambda$ is equal to $e_1-e_n$. Acting by
$W=S_n$, we get all vectors of the form $e_{i}-e_{j}$. Taking the
convex hull adds only $\bar 0$ to this set. We yield
$M(\lambda)=\{0\}\cup\{e_i-e_j\mid 1\le i,j \le n\}$. The saturation
property for this set was proved, e.g., in~\cite[Thm.1]{JM};
nevertheless, we present another proof based on a graph theory
approach.

\smallskip
\begin{theo}
Any subset in $M(\lambda)=\{0\}\cup\{e_i-e_j\mid 1\le i,j \le n\}$
is saturated.
\end{theo}
\begin{proof} Fix a nonempty subset $S\subseteq M(\lambda)$.
Construct a directed graph $\Gamma$ according to $S$. Let $\Gamma$
have $n$ vertices $A_1,A_2,\ldots,A_n$, where $A_i$ corresponds to
$e_i$. An arc $A_iA_j$ exists for each element $e_i-e_j$ from $S$.
Taken a formal linear combination $q_1v_1+q_2v_2+\ldots+q_sv_s$,
where $q_i\in\QQ$ and $v_i$ are the arcs of $\Gamma$, one can get a
corresponding linear combination $p_1e_1+p_2e_2+\ldots+p_ne_n$ by
substituting instead of all $v_i$ the corresponding vectors
$e_j-e_k$. If the coefficients $q_i$ are given, one can write the
explicit formula for $p_i$. Namely,
$$
p_i=\sum_{\textstyle{v_l=A_iA_m \atop \text{is an arc of
}\Gamma}}q_l-\sum_{\textstyle{v_l=A_mA_i \atop \text{is an arc of
}\Gamma}}q_l.\eqno(1)
$$
Our aim is to prove that $S$ is saturated. Reformulate the
saturation property in terms of graphs. We have to show that if
$v_1, v_2,\ldots,v_r$ is the set of edges of $\Gamma$, then each
vector $v$ from
$\QQ_+(v_1,v_2,\ldots,v_r)\cap\ZZ(v_1,v_2,\ldots,v_r)$ is also an
element of $\ZZ_+(v_1,v_2,\ldots,v_r)$. Take a vector
$v\in\QQ_+(v_1,v_2,\ldots,v_r)\cap\ZZ(v_1,v_2,\ldots,v_r)$. We can
multiply all its coordinates by an integer $m$ in such a way that
all coefficients of the corresponding $\QQ_+$-combination become
integer, i.e. $v_1=mv\in\ZZ_+(v_1,v_2,\ldots,v_r)\cap
m\ZZ(v_1,v_2,\ldots,v_r)$, where $m\ZZ(v_1,v_2,\ldots,v_r)$ denote
the set of linear combinations of $v_i$ where all coefficients are
divisible by $m$. Since $v_1\in m\ZZ(v_1,v_2,\ldots,v_r)$, we can
apply $(1)$ for the corresponding $m\ZZ$-combination and obtain that
all its coordinates are divisible by $m$. Now consider the
corresponding $\ZZ_+$-combination for $v_1$. If all vectors from
$\Gamma$ enter with multiplicities $\kratno m$, then we are done: we
can divide them by $m$ and get the required combination for $v$.
Otherwise, take all $v_i$ entering into the $\ZZ_+$-combination for
$v_1$ with multiplicities $\not\kratno m$ and name them
\textit{bad}. Let $\Gamma'$ be the subgraph in $\Gamma$ formed by
all bad edges. If we count the sum
$$
p_i=\sum_{\textstyle{v_l=A_iA_m \atop \text{is an arc of
}\Gamma'}}q_l-\sum_{\textstyle{v_l=A_mA_i \atop \text{is an arc of
}\Gamma'}}q_l\eqno(2)
$$
at a vertex $A_i$, then it will be divisible by $m$ since we
excluded only the edges with multiplicities $\kratno m$. This means
that $\Gamma'$ cannot have terminal vertices. Now find a cycle in
$\Gamma'$ (not necessarily oriented!). At each step, we will
increase by $1$ the values at the edges of this cycle oriented
clockwise and decrease by $1$ at the edges oriented
counter-clockwise. This operation does not change $v$ and $v_1$. In
several steps (not more than $m$) the value at an edge of $\Gamma'$
will be divisible by $m$. So the number of bad edges has decreased.
Repeating this operation, we will change the values at the edges in
such a way that they all will become divisible by $m$, and the
vector $v_1$ (and, also, $v$) will not change. Dividing all the
coefficients of the constructed $\ZZ_+$-combination by $m$, we get
the required $\ZZ_+$-combination for $v$.

\end{proof}

%%%%%%%%%%%%%%%%%%%%%%%%%%%%%%%%%%%%%%%%%%%%%%%%%%%%%%%%%%%%%%%%%%%%%%%%%%%%%%%%%%%%%%%%%%%%%%%%%%%%%%%%%%%%%%
\begin{definition}
We will mean by \textit{NSS} a non-saturated subset
$\{v_1,v_2,\ldots,v_s\}$ in the set $M$ of weights of a
representation. By  \textit{ENSS} we will mean the NSS together with
a vector $v$, where $v\in\ZZ (v_1, v_2, \dots, v_m) \cap \QQ_+ (v_1,
v_2, \dots, v_m)$, and $v\not\in\ZZ_+ (v_1, v_2, \dots, v_m)$.
\end{definition}

\subsection{The representation of $SL(2)$ with the highest weight $3\pi_1$}
\label{3pi} One has to verify the saturation property for each
subset in the set $M=\left\{\left(\frac32,
-\frac32\right),\left(\frac12, -\frac12\right),\left (-\frac12,
\frac12\right),\left(-\frac32, \frac32\right)\right\}$ (in the usual
basis). If there exists an NSS), then Lemma~\ref{l.o-} is
applicable, and this NSS can be reduced either to
$\left\{\left(\frac32, -\frac32\right),\left (\frac12,
-\frac12\right)\right\}$ or to $\left\{\left(\frac32,
-\frac32\right),\left(-\frac12, \frac12\right)\right\}$ (up to sign
change). But these subsets are both saturated. This means that the
initial NSS is also saturated, a contradiction.

\subsection{The representation of $SL(2)$ with the highest weight $4\pi_1$}
\label{4pi}Apply Lemma~\ref{l.o-} to the set
$M(\lambda)=\{(2,-2),(1,-1),(0,0),(-1,1),(-2,2)\}$. If an NSS
exists, then it contains not more than one vector of
$(1,-1),(-1,1)$, and not more than one of $(2,-2),(-2,2)$. So, up to
sign change, the NSS coincides with the set $\{(2,-2),(1,-1)\}$ or
with the set $\{(2,-2),(-1,1)\}$. But they are both saturated. We
get a contradiction.

\subsection{The representations of $SL(4)$, $SL(5)$ and $SL(6)$ with the highest weight $\pi_2$}
The set of weights $M(\lambda)$ is equal to $\{e_i+e_j\mid 1\le
i,j\le n, i\ne j\}$, where $n=4$, $5$, $6$. A graph $\Gamma$ with
$n$ vertices can be associated with any subset $S\subseteq
M(\lambda)$: an edge connecting $i$th and $j$th vertices exists
whenever $e_i+e_j\in S$. Suppose that there exists an ENSS
$\{w;v_1,\dots, v_m\,\mid \,v_i\in M(\lambda)\}$ :
\begin{gather*}
w=z_1v_1+\dots+z_mv_m=q_1v_1+\dots+q_mv_m,\quad z_i\in\ZZ,\,
q_i\in\QQ_+, \\
w\not\in\mathbb Z_+(v_1,\dots, v_m).
\end{gather*}
Consider all $v_i$ occurring into the right hand of this equality
with a nonzero coefficient $q_i$. By Lemma~\ref{l.oln}, we may
assume that they are linearly independent. To simplify the
reasoning, consider vector $v=w-\lfloor q_1\rfloor v_1-\dots-\lfloor
q_m\rfloor v_m$ instead of $w$. It is easy to see that $v$ belongs
to $\mathbb Z(v_1,\dots, v_m)$, to $\mathbb Q_+(v_1,\dots, v_m)$ and
does not belong to $\mathbb Z_+(v_1,\dots, v_m)$. We yield that
$\{v; v_1,\dots, v_m\}$ is also an ENSS. After this change all the
coefficients of the $\mathbb Q_+$-combination belong to $[0,1)$.

Construct a subgraph $\Gamma'\subset\Gamma$: take all the vertices
of $\Gamma$ and all the edges of $\Gamma$ entering into the $\mathbb
Q_+$-combination above with nonzero coefficients. Write the
coefficients of the $\QQ_+$-combination at the edges of $\Gamma'$.
The further proof consists of a search of all possible graphs
$\Gamma'$. The following observations will simplify the search.

(0.1) The number of edges in each connected component of $\Gamma'$
is not greater than the number of vertices (if not, the vectors
corresponding to the edges of this component will be linearly
dependent).

(0.2) The number of edges in $\Gamma'$ is less than the number of
vertices (it follows from $(*)$ that the dimension of the enveloping
space equals $n-1$).

(0.3) Graph $\Gamma'$ does not contain even cycles. It follows from
the fact that the edges of an even cycle are linearly dependent:
their alternating sum is null.

(0.4) It follows from~(0.1) and~(0.3) that each connected component
of~$\Gamma'$ either is a tree or contains exactly one cycle. In the
second case this cycle is always odd.

(0.5) It follows from~(0.2) that $\Gamma'$ has a vertex of degree
$0$ or $1$.

At each vertex, count the sum of all coefficients on the incident
edges, then for each sum take its fractional part. All these
fractional parts are equal due to the fact that all the sums in
vertices (they equal the coordinates of~$v$) become integer after
subtracting~$(*)$ with a proper coefficient. Now we conclude that

(0.6) $\Gamma'$ does not contain vertices of degree $0$ and $1$
simultaneously: if it does, the fractional parts of the sums in
vertices are all equal to $0$, but in the terminal vertex this sum
has only one summand and is not an integer.

We consider these two cases independently.
\smallskip

{\it Case 1.} Graph $\Gamma'$ has a vertex of degree $0$.

(1.1) Any other connected component of this graph is either a point
or has no terminal vertices (it follows from (0.6)). Moreover, it
follows from (0.4) that it is an odd cycle.

(1.2) We have $n\le 6$, consequently, the number of edges in
$\Gamma'$ is $\le 5$, but any odd cycle has $\ge 3$ edges, and we
yield that $\Gamma'$ has at most $1$ cycle.

Fulfill an exhaustive search within all graphs $\Gamma'$ having a
vertex of degree~$0$.
\begin{gather*}
n=4,\quad\text{graph is a cycle of length}\; 3,\\
n=5,\quad\text{graph is a cycle of length}\; 3,\\
n=6,\quad\text{graph is a cycle of length}\; 3,\\
n=6,\quad\text{graph is a cycle of length}\; 5.
\end{gather*}
The only possible~$\mathbb Q_+$-combination in these cases is
$\frac12(v_1+\dots+v_s)$. This means that $v$ is a sum of~$e_i$
corresponding to the vertices of the cycle. But it does not lie
in~$\ZZ(v_1, \dots, v_m)$ when $n$ is even. When $n=5$, consider
also the graph $\Gamma$. Since~$v$ is a $\ZZ$-combination of the
edges of~$\Gamma$, $\Gamma$ has more than $3$ edges:
$\Gamma\supset\Gamma'$, $\Gamma\neq\Gamma'$ and $\Gamma'$ has $3$
edges. In the representation above the sum of coefficients of~$v$ is
odd, hence we must apply~$(*)$ to the existing $\ZZ$-combination to
get the same representation. For this purpose the edges from
$\Gamma\setminus\Gamma'$ should touch all the vertices of $\Gamma$
(we name this property $(**)$).

In fig. 1 the graph $\Gamma'$ is drawn. To satisfy $(**)$, $\Gamma$
must contain at least the following edges (up to symmetry): see
fig.~2, 3 or 4. The vertices corresponding to $e_i$ are called
$V_i$. But in all cases we get a contradiction since $e_1+e_2+e_3$
is already a $\mathbb Z_+$-combination:
\smallskip
\begin{center}
\begin{tabular}{ccl}
in Fig. 2&& $e_1+e_2+e_3=V_1V_2+V_2V_3+V_1V_3+V_4V_5,$\\
in Fig. 3&& $e_1+e_2+e_3=V_4V_2+V_2V_1+V_1V_3+V_3V_5,$\\
in Fig. 4&& $e_1+e_2+e_3=V_4V_2+V_2V_5+2V_1V_3$.
\end{tabular}
\end{center}
\smallskip

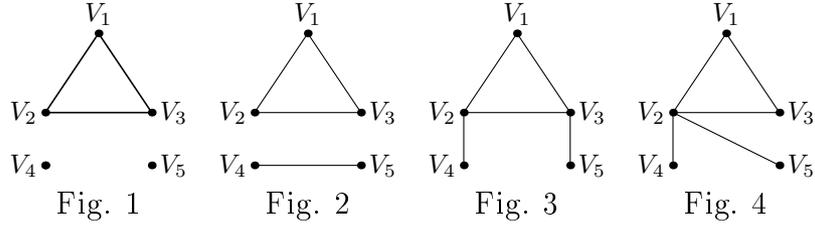
\begin{center}
\begin{tabular}{cccc}
\epsfig{figure=pictures.0} & \epsfig{figure=pictures.1} & \epsfig{figure=pictures.2} & \epsfig{figure=pictures.3}\\
Fig. 1&Fig. 2&Fig. 3&Fig. 4\\
&&&
\end{tabular}
\end{center}

We have shown that the counterexample does not exist in the case
when $\Gamma'$ has a vertex of degree $0$.

\smallskip
{\it Case 2.} Graph $\Gamma'$ has a vertex $X$ of degree $1$. Let
$XY$ be an edge incident to $X$. We need to subtract $(*)$ with the
same multiplicity as is at $XY$. We yield

(2.1) since $X$ is a terminal vertex of $\Gamma'$, either $XY$ is a
connected component of $\Gamma'$ or the vertex degree of $Y$ is $\ge
3$. Indeed, suppose that the vertex degree of $Y$ is $2$. Let $YZ$
be the second edge incident to $Y$, and let $q$ be the value written
at $YZ$. Then after subtracting $(*)$ the coefficient at $Y$ becomes
equal to $q$, but it must be integer, and we know that $q\in (0,
1)$. This is a contradiction.

Find all possible connected components of $\Gamma'$.

\begin{tabular}{clcl}
On $2$ vertices: & \epsfig{figure=smallpict.0}& &\\
On $3$ vertices: & \epsfig{figure=smallpict.1}&&\\
On $4$ vertices: & \epsfig{figure=smallpict.2}& and
&\epsfig{figure=smallpict.3}
\end{tabular}

Notice that if we have a connected component of $\Gamma'$ on $5$ or
$6$ vertices, it is the only connected component of $\Gamma'$. Using
this and (0.2), we obtain that $\Gamma'$ is a tree. Apply (2.1). We
have to consider only the following trees:

\begin{tabular}{clcl}
On 5 vertices: & \epsfig{figure=smallpict.4}&&\\
On 6 vertices: & \epsfig{figure=smallpict.5}& and
&\epsfig{figure=smallpict.6}
\end{tabular}

But the edges of \epsfig{figure=smallpict.7} are linearly dependent
(when $n=6$, one should sum all the thin edges, then subtract the
thick one, and obtain $(*)$). Therefore, this graph should not be
considered.

Fulfill an exhaustive search within all graphs $\Gamma'$ on $n$
vertices satisfying all the conditions above. In the case when one
of the connected components of $\Gamma'$ is a claw (i.e., all the
edges are incident to one vertex), its central vertex will
correspond to $e_1$ (it is easy to see that $\Gamma'$ cannot have
more than one claw).

\smallskip
\begin{center}
\begin{tabular}{|c|c|c|}
\hline
$n$&Splitting into& Permissible \\
&connected components&graphs\\
\hline
$4$&$2+2$&\epsfig{figure=smallpict.8}\\
\hline
$4$&$4$&\epsfig{figure=smallpict.3} or \epsfig{figure=smallpict.2}\\
\hline
$5$&$2+3$&\epsfig{figure=smallpict.10} \epsfig{figure=smallpict.1}\\
\hline
$5$&$5$&\epsfig{figure=smallpict.4}\\
\hline
$6$&$2+2+2$&\epsfig{figure=smallpict.10} \epsfig{figure=smallpict.10} \epsfig{figure=smallpict.10}\\
\hline
$6$&$2+4$&\epsfig{figure=smallpict.10} \epsfig{figure=smallpict.3} or \epsfig{figure=smallpict.10} \epsfig{figure=smallpict.2}\\
\hline
$6$&$6$&\epsfig{figure=smallpict.5}\\
\hline
\end{tabular}
\end{center}
\smallskip

The graphs \epsfig{figure=smallpict.8} and
\epsfig{figure=smallpict.3} do not satisfy our conditions: their
edges are linearly dependent.

If we start with \epsfig{figure=smallpict.2}, we can obtain only
$e_1$ as the $\QQ_+$-combination: all the three edges must appear in
the $\QQ_+$-combination with the same coefficient, let $a$,
$a\in(0,1)$. We sum these three vectors, obtain
$3ae_1+ae_2+ae_3+ae_4$, and subtract $(*)$ with a necessary
coefficient. Finally we obtain $2ae_1$. In this notation it already
has integer coordinates (equal to zero), this means that all the
other coordinates, $2a$ among them, must be integers, $a=\frac12$,
$v=e_1$. But $v$ cannot be obtained as the $\ZZ$-combination of the
vectors of the type $e_i+e_j$. Indeed, each $v_i$ has an even sum of
coordinates, $n$ is even, subtracting $(*)$ with an integer
coefficient does not change parity of the sum of coordinates, this
proves that any vector from $\ZZ_+\{v_i\}_{i=1}^m$ has an even sum
of coordinates.

The edges of graph \epsfig{figure=smallpict.10}
\epsfig{figure=smallpict.1}
 are linearly dependent (here $n=5$) because ($2\cdot$first edge + the sum of the edges of the cycle) $=0$.

Graph \epsfig{figure=smallpict.4}: using similar reasoning, $v=e_1$
or $2e_1$. But there exists an edge $\in\Gamma\setminus\Gamma'$,
hence $e_1$ is a $\mathbb Z_+$-combination of the edges of $\Gamma$:
take the sum of thick edges of \epsfig{figure=smallpict.9}.

The edges of \epsfig{figure=smallpict.10}
\epsfig{figure=smallpict.10} \epsfig{figure=smallpict.10} and
\epsfig{figure=smallpict.10} \epsfig{figure=smallpict.3}
 are linearly dependent.

In graph \epsfig{figure=smallpict.10} \epsfig{figure=smallpict.2}
 $v$ may be equal only to $e_1$, but $e_1$ can not be obtained as a
  $\mathbb Z$-combination: $6$ is even, the sum of coordinates of
   $e_1$ is odd.

In graph \epsfig{figure=smallpict.5} vector $v$ has to be
proportional to $e_1$, moreover, the coefficient must be even (we
use the reasoning as above, from the fact that $6$ is even it
follows that the sum of coordinates is even for any vector from
$\ZZ(v_1, \dots, v_m)$). But if we add any edge to this set, $2e_1$
will be obtained as a $\mathbb Z_+$-combination: take the sum of
thick edges of \epsfig{figure=smallpict.11}.

All the cases are considered, this completes the proof.

\subsection{The representation of $SL(3)$ with the highest weight $2\pi_1$}\label{224}
Its highest weight $\lambda$ is equal to $2\pi_1=2e_1$, and all the
weights of this representation are pointed in the figure below.

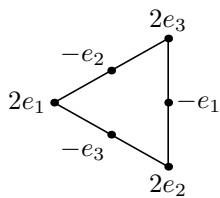
\begin{center}
\epsfig{figure=pictures.4}
\end{center}

Assume that this set contains an NSS $\{v_1, \dots, v_m\}$ (and ENSS
$\{v;v_1, \dots, v_m\}$). Consider the following possibilities.

If this NSS contains both $-e_1$ and~$-e_2$, then it does not
contain $-e_3$. Indeed, if it does, then $\ZZ_+(-e_1, -e_2,
-e_3)=\ZZ(-e_1, -e_2, -e_3)$, and this NSS cannot be non-saturated.
Using the similar reasoning, we get that it contains at most one of
the vectors $2e_1$ and $2e_2$: otherwise $\ZZ_+(-e_1, -e_2, 2e_1,
2e_2)=\ZZ(-e_1, -e_2, -e_3)$, and this NSS also cannot be
non-saturated. Consequently, this NSS coincides (up to the indices
renumbering) either with $(-e_1,-e_2,2e_1)$ or with
$(-e_1,-e_2,2e_1, 2e_3)$. But these subsets are saturated.

If the NSS has no vectors of form $-e_i$, then $\{v_1, \dots,
v_m\}\subset\{2e_1, 2e_2, 2e_3\}$. But in this case the NSS is also
saturated.

If the NSS has exactly one vector of form $-e_i$, suppose $-e_1$,
then fix a representation $n_1e_1+n_2e_2+n_3e_3$ for $v$,
$n_1,n_2,n_3\in\ZZ$. Then $n_2$ and $n_3$ have the same parity,
since $e_2$ and $e_3$ occur in the corresponding $\ZZ$-combination
for $v$ only as $2e_2$ and $2e_3$. Apply Lemma~\ref{l.oln}. We get a
representation $v=q_1v_1+q_2v_2$, where $v_1$ and $v_2$ are linearly
independent. We may suppose $q_i\in[0,1)$ because otherwise $v$ can
be changed to a vector $v-v_i$, which will diminish (in some sense)
the NSS. So the required combination will be either $\frac{2e_1}2$
or $\frac{2e_2+2e_3}2$. In the first case, recall that the NSS
contains $-e_1$, and we yield $e_1$ as the $\ZZ_+$-combination
$2e_1-e_1$. In the second case $-e_1$ is already the
$\ZZ$-combination, this is also a contradiction.

\subsection{The representation of $SL(6)$ with the highest weight $\pi_3$}
The highest weight $\lambda$ equals $e_1+e_2+e_3$,
$M(\lambda)=\{e_i+e_j+e_k\mid 1\le i<j<k\le 6\}$.

\begin{lemma}
In the notation above, for any $5$ linearly independent vectors
$v_1, \dots, v_5\in M(\lambda)$ the following equality holds:
$$
\ZZ(v_1, \dots, v_5)=\ZZ_{\equiv 0(3)}(e_1, \dots, e_6).\label{utv0}
$$
\end{lemma}

This means the following. For any vector $\mathbb Z(v_1, \dots,
v_5)$, its sum of coordinates is ${}\kratno 3$, and this property
does not depend on its representation. A surprising fact is that the
reverse statement is true~-- if we take an integer vector with the
sum of coordinates divisible by $3$, then the vector with the same
coordinates in quasi-basis can be obtained as a $\ZZ$-combination of
vectors $v_1, \ldots, v_5$.

\begin{proof}
First, we are going to show that vector $(1,1,1,0,0,0)$ lies in
$\mathbb Z(v_1, \dots, v_5)$. Notice that if we take $-v_i$ instead
of $v_i$ for an index $i$, then the $\ZZ$-lattice will not change.
In quasi-basis it means that we take a vector with the complementary
set of indices. Consequently, we may suppose that each $v_i$
contains $e_1$. Construct a graph with $5$ vertices and $5$ edges:
each vertex corresponds to one of the integers $2$, $3$, $4$, $5$,
$6$, vertices $i$ and $j$ are connected with an edge if and only if
there exists $k$, such that $v_k=e_1+e_i+e_j$. Examine all possible
graphs. Vectors $\{v_i\}_{i=1}^5$ are linearly independent by the
data, this fact has the following conclusions:

(1) the graph has no cycles of length $4$~-- otherwise we have a
zero sum of form $v_1-v_2+v_3-v_4$;

(2) the graph has no vertices of degree $0$~-- otherwise a subgraph
containing other $4$ vertices has $5$ edges, this means that it has
a cycle of length $4$.

(3) this graph is connected. Indeed, if it has $2$ or more connected
components, none of which is a single vertex, then it has two
connected components of $2$ and $3$ vertices respectively, which
gives at most $3+1=4$ edges. This graph has a cycle since it has a
sufficient number of edges. Using~(1), we get that this cycle has
$3$ or $5$ vertices, which means that it is odd.

For any vertex $X$ of this graph there exists an odd cycle (maybe
not a circuit) passing through this vertex. Indeed, if $X$ is
already a vertex of the odd cycle constructed above, then we are
done. Otherwise the required cycle has three parts. The first part
is a path from $X$ to any vertex $Y$ of the odd cycle, the second is
the odd cycle, the third is the reverse path from $Y$ to $X$.

Now we show that any two vertices can be connected by an odd path
(not necessarily simple). Indeed, take two arbitrary vertices and
connect them with an arbitrary path. This path is either odd or
even. If it is odd, we are done. If it is even, we can combine it
with an odd cycle passing through the first vertex of this path.

Now we explain how to obtain $(1,1,1,0,0,0)$. This vector
corresponds to the pair $(2,\,3)$ of vertices of the graph. If they
are already connected by an edge, we are done. Otherwise connect
them with an odd path and take the alternating sum of its edges.

To finish the proof, explain how to obtain an arbitrary vector with
the sum of coor\-di\-na\-tes~${}\kratno 3$. We can obtain
$(0,1,1,1,0,0)$ in the same way as we obtain $(1,1,1,0,0,0)$. Hence
we can easily obtain $(1,0,0,-1,0,0)$ (and all the other vectors
which can be obtained from this by a permutation of coordinates) as
their difference. To get the required decomposition for an arbitrary
vector, we will successively subtract vectors of form $e_i-e_j$ from
our vector. At each step choose $i$ and $j$ such that the $i$th
coordinate of our vector is maximal and the $j$th coordinate is
minimal. If the difference between the maximal and the minimal
coordinates is $\geqslant 2$, then in several steps it will
diminish. If it equals $1$, then the vector is equal to
$(e_i+e_j+e_k)+a(1,1,1,1,1,1)$, $a\in\ZZ$, but this vector equals
$(e_i+e_j+e_k)\in\ZZ(v_1,_2,\ldots,v_5)$. If this difference equals
$0$, then the remaining vector equals $0$. In both cases we are
done.
\end{proof}

Reformulate Lemma~\ref{utv0}:
\begin{lemma}
In the notation above, for any $m$ linearly independent vectors
$\{v_1, \dots,v_m\}\subseteq M(\lambda)$ the following equality
holds:
$$
\mathbb Z(v_1, \dots, v_m) = \QQ(v_1, \dots, v_m)\cap\ZZ_{\equiv
0(3)}(e_1, \dots, e_6).\label{utv1}
$$
\end{lemma}
\begin{proof} If $m < 5$, we add several vectors from $M(\lambda)$ to this set to get a
set of $5$ linearly independent vectors $v_1, \dots,v_5$.  Apply
Lemma~\ref{utv0} to the set $v_1, \dots,v_5$. Since a vector has a
unique representation on a basis,
$$
\mathbb Z(v_1, \dots,v_m)=\langle v_1, \dots,v_m\rangle\cap \mathbb
Z(v_1, \dots,v_5).
$$
Finally, using Lemma~\ref{utv0}, we get that $\mathbb Z(v_1,
\dots,v_5)$ coincides with $\ZZ_{\equiv 0(3)}\{e_1, \dots, e_6\}$.
\end{proof}

\begin{remark}
One may suppose that if we omit $(*)$ and take linearly independent
vectors $v_1, \ldots, v_6\in\QQ^6$, $v_i=e_p+e_q+e_r$, then
$\ZZ(v_1, \ldots, v_6)=\ZZ_{\equiv 0(3)}(e_1,\ldots, e_6 )$.
However, this is not true~-- if we take the following vectors, then
the volume of the unit cube will be equal to $6$, and the index of
the new lattice in $\ZZ^6$ will be $6$, not $3$.
$$
\left(
    \begin{array}{c}
    v_1\\
    v_2\\
    v_3\\
    v_4\\
    v_5\\
    v_6\\
    \end{array}
\right) = \left(
    \begin{array}{lr}
    1\;1\;1\;  0\;0  \;0\\
    1\;0\;0\;  1\;1  \;0\\
    0\;1\;0\;  0\;1  \;1\\
    0\;0\;1\;  1\;0  \;1\\
    1\;1\;0\;  0\;0  \;1\\
    1\;1\;0\;  1\;0  \;0\\
    \end{array}
\right),\quad \det\left(
    \begin{array}{c}
    v_1\\
    v_2\\
    v_3\\
    v_4\\
    v_5\\
    v_6\\
    \end{array}
\right) = 6.
$$
\end{remark}
\smallskip

Now we have to show that every subset $\{v_i\}$ in $M(\lambda)$ is
saturated. To the contrary, let $\{v_1, \dots, v_m\}\subset
M(\lambda)$ be an NSS. Then there exists a vector $v$ and (by
Lemma~\ref{l.oln} and after renumbering) a linearly independent
subset $\{v_1, \dots, v_s\}\subseteq\{v_1, \dots, v_m\}$ such that
$$
v=q_1v_1+\dots+q_sv_s=z_1v_1+\dots+z_mv_m,\qquad
q_i\in\QQ_+,\,z_i\in\ZZ.
$$
This gives $v\in\ZZ_{\equiv 0(3)}(e_1,\dots,e_6)$, and $v\in\langle
v_1, \dots, v_s\rangle$. Using Lemma~\ref{utv1}, we obtain that
${v\in\ZZ(v_1, \dots, v_s)}$, $v=z_1'v_1+\dots+z_s'v_s$. Since $v_1,
\dots, v_s$ are linearly independent, for any $i$ we have
$q_i=z_i'\in \ZZ_+$, and $v\in\ZZ_+(v_1, \dots, v_s)$. This shows
that each subset of $M(\lambda)$ is saturated.

\section{Negative results}
Let $\lambda$ be a highest weight not listed in the Main Theorem.
One has to construct an NSS in~$M(\lambda)$. There are two
opportunities for $\lambda$: either the absolute values of all its
usual coordinates are $<1$, or $\lambda$ has a coordinate with the
absolute value $\geqslant 1$. Speaking informally, the second case
is practically always the consequence of the first one
(Lemma~\ref{l.ovl}), but the NSS in the first case is constructed
recursively and its capacity increases when $n$ increases. The
construction of the second case gives an NSS of only $4$ vectors for
any $n$.

To prove that a set $\{v_1, \dots, v_m\}$ is not saturated, we will
construct a so-called {\it discriminating function} $f(v)$ with the
following properties: linearity, $=0$ when $v=e_1+\dots+e_n$ (to
make it correctly defined), and non-negativity on the vectors of the
set $\{v_1, \dots, v_m\}$. The discriminating function will be
applied as follows. If we want to show that $\{v_0;v_1, \dots,
v_m\}$ is an ENSS, it suffices to present the corresponding $\QQ_+$-
and $\ZZ$-combinations for $v_0$ and construct a discriminating
function $f$, such that $f(v_0)$ cannot be composed as the sum of
$f(v_i)$ with {$\ZZ_+$-coefficients}.

Further, $x_i$ denotes the function of taking the $i$th coordinate
of a vector in some quasi-basis representation.

\subsection{The fundamental weights.} In this case $\lambda$ equals
$$\pi_k=\pi_{k,n}=\left(\frac{n-k}{n},
\dots,\frac{n-k}{n}, -\frac{k}{n}, \dots, -\frac{k}{n} \right)
$$
in the usual basis, $0<k<n$, $n \geqslant 3$ (when $n=2$, the
corresponding representation is mentioned in the Main Theorem). In
some proofs we will consider $\pi_k$ for $SL(n)$'s of different
dimensions simultaneously, so the second index in the notation
$\pi_{k,n}$ carries this data. Here $M(\lambda) = \{\sigma
\lambda\mid \sigma \in S_n\}$. The highest weight is equal to
$e_1+e_2+\ldots+e_k$ in quasi-basis, all the points of $M(\lambda)$
have a form $ e_{i_1}+e_{i_2}+\ldots+e_{i_k}$, $1\le i_1<i_2<
\ldots<i_k\le n$.

Now we can reformulate the problem. Let $\{e_i\}$ be the
quasi-basis, $k<n$, the weight $\lambda=\pi_k$ is not listed in the
Main Theorem. One has to find a non-saturated subset in the set
$$
\{e_{i_1}+e_{i_2}+\ldots+e_{i_k} \mid  1\le i_1<i_2< \ldots<i_k\le n\}.
$$

The construction will use induction on $n$. In the next section we
will produce the NSSes which will be the base of the induction.

\subsubsection{Important particular cases.}

Example 1. $n=7, k=2$. The NSS will consist of those and only those
vectors which are the sums of two quasi-basis vectors connected with
an edge in the graph below. We have
\begin{gather*}
v= e_1+e_2+e_3 = \frac{1}{2}\Bigl((e_1+e_2)+(e_2+e_3)+(e_1+e_3)
\Bigr),\\
v=-(e_4+e_5+e_6+e_7)=2(e_2+e_3)-(e_2+e_4)-(e_2+e_5)-(e_3+e_6)-(e_3+e_7).
\end{gather*}
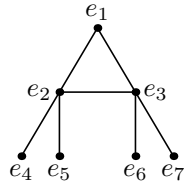
\begin{center}
\epsfig{figure=pictures.5}
\end{center}

\begin{gather*}
\text{Let}\quad f=5(x_2+x_3)-2(x_1+x_4+x_5+x_6+x_7).\quad\text{Then}\\
f(e_2+e_3)=10,\\
f(e_1+e_2)=f(e_1+e_3)=f(e_2+e_4)=f(e_2+e_5)=f(e_3+e_6)=f(e_3+e_7)=3,\\
f(v)=f(e_1+e_2+e_3)=5\cdot 2-2=8.
\end{gather*}
It is clear that $8$ cannot be represented as a sum where each
summand equals either $3$ or $10$.

\smallskip

Example 2. $n=8, k=3$. Consider the following vectors (in
quasi-basis):
$$
\left(
    \begin{array}{c}
    v_1\\
    v_2\\
    v_3\\
    v_4\\
    v_5\\
    \\
    v_6\\
    v_7\\
    v_8
    \end{array}
\right) = \left(
    \begin{array}{lr}
    0\;0\;1\;1\;1 & 0\;0\;0\\
    1\;0\;0\;1\;1 & 0\;0\;0\\
    1\;1\;0\;0\;1 & 0\;0\;0\\
    1\;1\;1\;0\;0 & 0\;0\;0\\
    0\;1\;1\;1\;0 & 0\;0\;0\\
    \\
    0\;0\;1\;1\;0 & 1\;0\;0\\
    0\;1\;0\;1\;0 & 0\;1\;0\\
    0\;1\;1\;0\;0 & 0\;0\;1\\
    \end{array}
\right).
$$
\smallskip
Take
$v=(1,1,1,1,1,0,0,0)=\frac{1}{3}(v_1+v_2+v_3+v_4+v_5)=2v_5-v_6-v_7-v_8$.
\smallskip
Let $f=x_1+5(x_2+x_3+x_4)+2x_5-6(x_6+x_7+x_8)$. Then
\begin{gather*}
f(v_1)=12,\quad f(v_2)=f(v_3)=8,\quad f(v_4)=11,\\
f(v_5)=15,\quad f(v_6)=f(v_7)=f(v_8)=4,\quad f(v)=18.
\end{gather*}
It is easy to see that $18$ cannot be represented as a sum where
each summand equals $4$, $8$, $11$, $12$, or $15$.

\smallskip

Example 3. $n=2k, k\geqslant 4$.
$$
    \left(
    \begin{array}{c}
        v_1\\
        v_2\\
        v_3\\
        \vdots \\
        v_{k-1}\\
        v_k\\
        \\
        v_{k+1}\\
        v_{k+2}\\
    \end{array}
\right) =
 \left(
    \begin{array}{ccccccccccccc}
        1&0&0&\ldots&0&0&                      &0&1&1&\ldots&1&1\\
        0 &1&0&\ldots&0 &0&                    &1&0&1&\ldots&1&1\\
        0 &0&1&\ldots&0 &0&                    &1&1&0&\ldots&1&1\\
        \vdots&\vdots&\vdots&\ddots&\vdots&\vdots&    &\vdots&\vdots&\vdots&\ddots&\vdots&\vdots\\
        0 &0&0&\ldots&1 &0&                    &1&1&1&\ldots&0&1\\
        0 &0&0&\ldots&0 &1&                    &1&1&1&\ldots&1&0\\
        \\
        0 &1&0 &\ldots&0 &0&   &1 &1 &1&\ldots &1 &0\\
        1 &1&0 &\ldots& 0 &0&  & 0 &1 &1&\ldots &1 &0\\
\end{array}
 \right)
$$

\smallskip
Show that it is an NSS. Let
$v=(\underbrace{0,\dots,0}_k,\underbrace{1,\dots,1}_k)$,
$v=v_1+v_{k+1}-v_{k+2}$,
\begin{multline*}
\frac{1}{k-2}(v_1+\dots+v_k)=\frac{1}{k-2}(\underbrace{1,\dots,1}_k,\underbrace{k-1,\dots,k-1}_k)=\\
=\frac{1}{k-2}(\underbrace{0,\dots,0}_k,\underbrace{k-2,\dots,k-2}_k)=
(\underbrace{0,\dots,0}_k,\underbrace{1,\dots,1}_k)=v.
\end{multline*}
To explain why $v$ is not a $\mathbb Z_+$-combination of vectors
$v_i$, consider two cases.

First case, when $k=4$, let $f=-6x_3-7x_4+5(x_5+x_6+x_7)-2x_8$. Then
$$
f(v_1)=f(v_2)=8,\,f(v_3)=2,\,f(v_4)=8,\,f(v_5)=15,\,f(v_6)=10,\,f(v)=13.
$$
But it is easy to see that $13$ cannot be represented as a sum where
each summand equals $2$, $8$, $10$, or $15$.

Second case, when $k\geqslant 5$, let
$f=(k-2)(x_{k+1}+\dots+x_{2k})-k(x_3+\dots+x_k)$. Then
\begin{gather*}
f(v_1)=f(v_2)=(k-2)(k-1),\\
f(v_3)=f(v_4)=\dots=f(v_k)=(k-1)(k-2)-k,\\
f(v_{k+1})=(k-2)(k-1),\\
f(v_{k+2})=(k-2)^2,\\
f(v)=k(k-2).
\end{gather*}
If $k\geqslant 6$, then two least possible summands give too much:
$2((k-1)(k-2)-k)>k(k-2)$, if $k=5$, then $15$ should be represented
as a sum where each summand equals $12$, $7$, or $9$, but this is
impossible.

\smallskip

Example 4. $n=2k+1, k\geqslant 3$.

$$
\left(
    \begin{array}{c}
        v_1\\
        v_2\\
        v_3\\
        \vdots \\
        v_k\\
        v_{k+1}\\
        \\
        v_{k+2}\\
        v_{k+3} \\
        \vdots\\
        v_{2k}\\
        v_{2k+1}\\
    \end{array}
\right) =
 \left(
    \begin{array}{cccccccccccc}
        0 & 1 & 1 &\ldots & 1 & 1 && 0 & 0&\ldots & 0 & 0\\
        1 & 0 & 1 &\ldots & 1 & 1 && 0 & 0&\ldots & 0 & 0\\
        1 & 1 & 0 &\ldots & 1 & 1 && 0 & 0&\ldots & 0 & 0\\
        \vdots&\vdots&\vdots&\ddots&\vdots&\vdots&&\vdots&\vdots&\ddots&\vdots&\vdots\\
        1 & 1 & 1 &\ldots & 0 & 1 && 0 & 0&\ldots & 0 & 0\\
        1 & 1 & 1 &\ldots & 1 & 0 && 0 & 0&\ldots & 0 & 0\\
        \\
        0 & 1&\ldots & 1 & 1 &  0 && 1 & 0&\ldots & 0 & 0\\
        1 & 0&\ldots & 1 & 1 &  0&& 0 & 1&\ldots & 0 & 0\\
        \vdots&\vdots&\ddots&\vdots&\vdots&\vdots&&\vdots&\vdots&\ddots&\vdots&\vdots\\
        1 & 1&\ldots & 0 & 1 &  0 && 0 & 0&\ldots & 1 & 0\\
        1 & 1&\ldots & 1 & 0 &  0 && 0 & 0&\ldots & 0 & 1\\
    \end{array}
\right)
$$

\bigskip

Let $v=(\underbrace{1,\dots,1}_{k+1},\underbrace{0,\dots,0}_k)$.
Then
$$
v=\frac{1}{k}(v_1+\dots
+v_{k+1})=\frac{1}{k}(\underbrace{k,\dots,k}_{k+1},\underbrace{0,\dots,0}_k)=
(\underbrace{1,\dots,1}_{k+1},\underbrace{0,\dots,0}_k),
$$

\begin{multline*}
(k-1)v_{k+1}-v_{k+2}-\dots-v_{2k+1}=\\
=(k-1)(\underbrace{1,\dots,1}_k,\underbrace{0,\dots,0}_{k+1})
-(\underbrace{k-1,\dots,k-1}_k,0,\underbrace{1,\dots,1}_k)=\\
=(\underbrace{0,\dots,0}_{k+1},\underbrace{-1,\dots,-1}_k)
=(\underbrace{1,\dots,1}_{k+1},\underbrace{0,\dots,0}_k)=v.
\end{multline*}
It suffices to show that $v$ does not belong to
$\ZZ_+(v_1,v_2,\ldots,v_{2k+1})$. Let
$f=(k+1)(x_1+\dots+x_k)-k(x_{k+1}+\dots+x_{2k+1})$. Then
\begin{gather*}
f(v_1)=\dots=f(v_k)=k^2-k-1,\\
f(v_{k+1})=k(k+1),\\
f(v_{k+2})=\dots=f(v_{2k+1})=k^2-k-1,\\
f(v)=k^2.
\end{gather*}
But if $k\geqslant 3$, then $k^2<2(k^2-k-1)$, so $k^2$ cannot be
represented as a sum where each summand equals either $(k^2-k-1)$ or
$k(k+1)$. This means that $v\not\in\ZZ_+(v_1,v_2,\ldots,v_{2k+1})$.

\smallskip

Example 5. $n=8, k=2$.

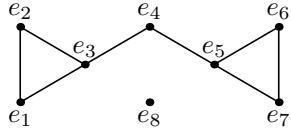
\begin{center}
\epsfig{figure=pictures.6}
\end{center}

The NSS will contain those and only those vectors which are sums of
two quasi-basis vectors connected with an edge in the graph above.
Let $v=e_1+e_2+e_3+e_5+e_6+e_7$. Then
$$
v =
\frac{1}{2}\left((e_1+e_2)+(e_2+e_3)+(e_1+e_3)+(e_5+e_6)+(e_6+e_7)+(e_5+e_7)
\right),
$$
$$
v = (e_1+e_2)+(e_3+e_4)-(e_4+e_5)+(e_5+e_6)+(e_5+e_7).
$$

Check that $e_1+e_2+e_3+e_5+e_6+e_7$ cannot be represented as a
$\ZZ_+$-combination of the vectors of our set. Let
${f=x_1+x_2+x_3+2(x_5+x_6+x_7)+9x_4-18x_8}$. Then
\begin{gather*}
f(e_1+e_2)=f(e_2+e_3)=f(e_1+e_3)=2,\quad
f(e_5+e_6)=f(e_6+e_7)=f(e_5+e_7)=4,\\
f(e_3+e_4)=10,\quad f(e_4+e_5)=11,\quad f(v)=9.
\end{gather*}
But $9$ cannot be represented as the sum of integers $2$, $4$, $10$,
or $11$.

\smallskip

Example 6. $n=9, k=3$.

Consider the following vectors:
\begin{gather*}
v_1=e_1+e_2+e_4,\\
v_2=e_1+e_2+e_5,\\
v_3=e_2+e_3+e_6,\\
v_4=e_2+e_3+e_7,\\
v_5=e_1+e_3+e_8,\\
v_6=e_1+e_3+e_9,\\
v_7=e_2+e_4+e_6.
\end{gather*}
Then $v=e_1+e_2+e_3=\frac13(v_1+v_2+v_3+v_4+v_5+v_6)=v_1+v_3-v_7$.
\smallskip

Check that $v$ is not a $\ZZ_+$-combination of $v_1$, $v_2$, $v_3$,
$v_4$, $v_5$, $v_6$, and $v_7$. Let
$f=5(x_1+x_2+x_3+x_4)-4(x_5+x_6+x_7+x_8+x_9)$. Then
\begin{gather*}
f(v_1)=15,\\
f(v_2)=f(v_3)=f(v_4)=f(v_5)=f(v_6)=f(v_7)=6,\\
f(v)=15.
\end{gather*}
Note that $v\ne v_1$ and $f(v_1)=f(v)$, so we conclude that if
$v\in\mathbb Z_+(v_1, \ldots, v_7)$ then $v_1$ does not occur in
this decomposition. But  $15\notkratno 6$, this means that $v$
cannot be obtained as a $\ZZ_+$-combination of $v_i$'s.

\smallskip

Example 7. $n=10, k=4$. Consider the following vectors:
\begin{gather*}
v_1=e_1+e_2+e_3+e_5,\\
v_2=e_1+e_2+e_4+e_6,\\
v_3=e_3+e_4+e_5+e_6,\\
v_4=e_5+e_6+e_7+e_8,\\
v_5=e_5+e_7+e_8+e_9,\\
v_6=e_6+e_7+e_8+e_{10},\\
v=e_1+e_2+e_3+e_4+e_5+e_6=\frac12(v_1+v_2+v_3)=v_4-v_5-v_6.
\end{gather*}

Show that $v\not\in\mathbb Z_+(v_1, \dots, v_6)$. Let
$f=x_1+x_3+x_4+6x_7+6x_8-7x_9-8x_{10}$. Then
\begin{gather*}
f(v_1)=f(v_2)=f(v_3)=2,\\
f(v_4)=12,\,f(v_5)=5,\,f(v_6)=4,\\
f(v)=3.
\end{gather*}
But it is clear that $3$ cannot be represented as a sum where each
summand equals $2$, $4$, $5$, or~$12$.

\subsubsection{Case when $n\not\kratno k, n\not\kratno (n-k)$ .}

It follows from these two conditions that $n\geqslant 5$. The
exceptional case $\frac{k}{n}\in\{\frac{2}{5}, \frac{3}{5}\}$ will
be considered at the end of the section. Further we (temporarily)
suppose that $\frac{k}{n}\not\in\{\frac{2}{5}, \frac{3}{5}\}$, which
gives $n\geqslant 7$.

\begin{lemma}
Suppose that there exists an NSS for a pair $(n,k)$, where $(n,k)$
satisfy the conditions above. Then for each $r\in\NN$, there exists
an NSS for the pair $(nr,kr)$. \label{l.oumn}
\end{lemma}
\begin{proof}
Consider an arbitrary vector from $M(\pi_{k,n})$. Write down its
quasi-coordinates $r$ times in succession. The result is a vector
from $M(\pi_{kr,nr})$: it has $kr$ $1$'s and $(n-k)r$ $0$'s. If one
takes an NSS for $(n,k)$ and performs this procedure on each vector,
the result will be an NSS for $(nr,kr)$.
\end{proof}

Thus, if we construct an NSS for all pairs $(n,k)$ where
$\gcd(n,k)=1$, then the NSS for all other pairs will be also
constructed according to the Lemma.

\begin{lemma}[the Step procedure]
Having constructed an NSS for a pair $(n,k)$, one can construct an
NSS for the pair $(n+k,k)$ according to the existing
NSS.\label{l.oshage}
\end{lemma}

\begin{proof}
The keypoint is that if one takes a weight from $M(\pi_{k,n})$,
writes it down in the form where all its quasi-coordinates are equal
to $0$ or to $1$, and adds $k$ coordinates equal to $0$, then this
weight can be considered as a weight from $M(\pi_{k,n+k})$. If we
start with an ENSS $(v; v_1, v_2, \dots, v_m)$ for $(n,k)$, we
should perform this procedure on all its vectors  and then add one
more vector $v_{m+1}$ which has $0$s at the first $n$ digits and
$1$s at the $k$ adjoint digits. Now we show that the obtained set
$(v'; v'_1,\ldots,v'_m,v_{m+1})$ is indeed an NSS in
$M(\pi_{k,n+k})$.

Suppose that a vector $v$ lies in the ENSS for $(n,k)$,
$v=q_1v_1+q_2v_2+\ldots+q_sv_s$, $q_i\in\QQ_+$. If we fix some
representations for $v$ and for all $v_i$ in quasi-basis, then this
equality can be re-written in the \textit{formal} basis in the
following form:
$$
v=q_1v_1+q_2v_2+\ldots+q_sv_s-\alpha(f_1+f_2+\ldots+f_n),
$$
where $f_i$'s are the counter images of $e_i$'s under the projection
$\QQ^n\rightarrow\mathfrak X(T)\otimes_{\ZZ}\QQ$,
$(q_1,q_2,\ldots,q_n)\mapsto q_1e_1+q_2e_2+\ldots+q_ne_n$.
Obviously, $f_i$'s are linearly independent. For each $v_i$, fix a
representation in which it has $k$ coordinates equal to~$1$ and
$n-k$~coordinates equal to~$0$. The vector $v$ is nonzero,
consequently, it has a representation where all its coordinates are
nonnegative, but some of them are zeroes. Fix this representation.
Then $\alpha\geq 0$ (otherwise all coordinates of $v$ are strictly
positive), and we get that in $\QQ^{n+k}$ the following equality
holds:

$$
v'=q_1v'_1+q_2v'_2+\ldots+q_sv'_s+\alpha(f_{n+1}+\dots
+f_{n+k})-\alpha(f_1+f_2+\ldots+f_{n+k}).
$$
This shows that $v'$ lies in the $\QQ_+$-lattice generated by
$v'_1,\ldots,v'_m,v_{m+1}$ (here all vectors taken in quasi-basis
$\{e_1,\ldots,e_{n+k}\}$).

Similarly one can show that $v'$ still lies in the
$\ZZ(v'_1,\ldots,v'_m,v_{m+1})$.

To prove that the constructed set is indeed an ENSS, it remains to
show that $v'$ does not lie in $\ZZ_+(v'_1,\ldots,v'_m,v_{m+1})$.
Suppose the contrary. Let $v'\in\ZZ(v'_1,\ldots,v'_m,v_{m+1})$. Omit
last $k$ coordinates. We get that $v\in\ZZ_+(v_1,\ldots,v_m)$, so
$\{v; v_1,v_2,\ldots,v_m\}$ is not an ENSS for $(n,k)$.
\end{proof}

Now we can explain how, using these Lemmas, the NSS's can be
constructed for all pairs $(n,k)$, for which the following three
conditions are held:

$(1)\; 1<k<n-1$,

$(2)\; \gcd(n,k)=1$,

$(3)\; n\geqslant 7$.

Use descent on $n$. Suppose the NSSes are constructed for all pairs
$(m,l)$, satisfying the conditions above, with $m<n$. Take a pair
$(n,k)$. Suppose $k<\frac{n}{2}$ (otherwise change it by $n-k$ and
seek for an NSS for the pair $(n,n-k)$, the case $n=2k$ is
impossible because $\gcd(n,k)=1$). If all the conditions are held
for the pair $(n-k,k)$, then we have an NSS for it, and using the
Step procedure, this NSS can be re-made into the NSS for $(n,k)$.
Let us find all the cases when at least one of the conditions fails
for $(n-k,k)$.

Condition $(1)$ fails iff $n=2k+1$. But we have $n\geqslant 7$, then
$k\geqslant 3$. In this case we already have an NSS (example 4).

Condition $(2)$ never fails.

Condition $(3)$ fails iff $n-k\le 5$. Find these cases. Recall that
$k\le (n-1)/2$. Substitute it: $n\le (n-1)/2 +5$. This gives $n\le
9$. List all these pairs $(n,k)$ (with $k<\frac{n}{2}$).
\begin{gather*}
n=7. \quad\text{Pairs}\; (7,2)\text{ and }(7,3).\\
n=8. \quad\text{Pair}\; (8,3).\\
n=9. \quad\text{Pairs}\; (9,2)\text{ and }(9,4).
\end{gather*}

But we already have NSSes for all these pairs. Indeed, cases $(7,2)$
and $(8,3)$ coincide with Examples 1 and 2, respectively. Cases
$(7,3)$ and $(9,4)$ are the particular cases of $n=2k+1$ (Example
4). Case $(9,2)$ can be obtained from $(7,2)$ (Example 1) using the
Step procedure.

Finally, take all the cases where the NSS is already constructed as
the base of the descent. In all the other cases the descent is
feasible, consequently, we have constructed an NSS for all pairs
$(n,k)$ for which the conditions $(1)\,-\,(3)$ hold.

Now we consider the case $\frac{k}{n}\in\left\{\frac{2}{5},
\frac{3}{5}\right\}$. Let $k=2k_1$, $n=5k_1$, $k_1\geqslant 2$. When
$k_1\geqslant 4$, we can construct an NSS using the Step procedure
and substitution $k\rightarrow n-k$: starting with an NSS for
$(2k_1,k_1)$, we successively construct NSSes for $(3k_1,k_1)$,
$(3k_1,2k_1)$, and $(5k_1,2k_1)$. When $k_1=2$, the Example~7 can be
applied.

When $k_1=3$, the pair $(n,k)=(15,6)$, and the required NSS can be
obtained from Example~6 using the Step procedure.

\subsubsection{Case when $n\kratno k$ or $n\kratno (n-k)$.}
Assume that $k\leq n/2$. Then $n\kratno k$, $n=kd$. In the case when
$k=1$ all the subsets in the sets of weights are saturated
(see~\ref{tavt}), further $k\geqslant 2$.

When $k\geqslant 4$, Example~3 shows that the NSS exists for the
pair $(2k,k)$. Using the Step procedure, we can easily rebuild this
NSS into the NSS for a pair $(kd,d)$, where $d\geqslant 2$. It
remains to consider cases $k=2$ and $3$.

$k=2$. It follows from the Main Theorem that $d\geqslant 4$. But we
already have an NSS for the pair $(8,2)$ (Example~5), using the Step
procedure, we can construct NSSes for all $d>4$.

$k=3$. We already have an NSS for $(9,3)$, using the Step procedure,
we can construct an NSS for all $n\kratno 3$, $n>9$.

We are done.

\subsection{Non-fundamental weights.}

In the previous section the structure of $M(\lambda)$ was much
easier than in the general case. Indeed, by definition
$M(\lambda)=(\lambda+\Phi )\cap P$, where $P=\conv\{\sigma
\lambda\mid \sigma \in S_n\}$. If $\lambda$ is a fundamental weight,
then we just take $\{\sigma \lambda\mid \sigma \in S_n\}$, not
dealing with $\Phi$, because it does not add new points. However, if
$\lambda$ is not fundamental, then $M(\lambda)$ contains internal
points of $P$, and this is a great advantage for constructing NSSes.

\begin{lemma}[Inclusion Lemma]
Let $\lambda$ and $\lambda'$ be two dominant weights, such that
$\lambda'\in M(\lambda)$, and there exists an NSS in $M(\lambda')$.
Then there exists an NSS in $M(\lambda)$. \label{l.ovl}
\end{lemma}

\begin{proof}
Notice that $\forall\sigma\in W$ $\;\sigma \lambda'\in M(\lambda)$
and $P'=\conv\{\sigma \lambda'\mid  \sigma \in W\}\subset P$. Then
$M(\lambda')=(\lambda'+ \Phi )\cap P'\subset M(\lambda)$. This means
that the NSS for $\lambda'$ is also an NSS for $\lambda$.
\end{proof}

There are two cases: the first one~-- all the usual coordinates of
$\lambda$ are integer, and the second one~-- all of them are
non-integer. Consider these cases independently. The coordinates of
vectors in the usual basis will be denoted by $y_i$. The
discriminating functions in the usual basis must satisfy only the
conditions of linearity and $f(v_i)\geqslant 0$.

\subsubsection{All the coordinates of $\lambda$ are integer,
$\lambda\ne(1,0,\dots,0,-1)$.}\label{321}
\begin{definition} By a {\it Shift} we will denote the following procedure: take a point $\lambda=(y_1, \ldots, y_n)$,
fix two indices $i<j$ such that $|y_i-y_j|\geqslant 2$ and replace
$\lambda$ with the point $\lambda'$, where
$\lambda'=(\dots,y_i-1,\dots,y_j+1,\dots)$ if $y_i>y_j$ and
$(\dots,y_i+1,\dots,y_j-1,\dots)$ otherwise.
\end{definition}

The point $\lambda'$ lies in $M(\lambda)$. Indeed, $M(\lambda)$
contains the point $(\dots,y_j,\dots,y_i,\dots)$, its convex hull
with $\lambda$ (with the proper coefficient) contains $\lambda'$.
Notice that after each Shift $y_1^2+\ldots+y_n^2$ diminishes by a
positive integer. Indeed, let $x=\max\{y_i, y_j\}$, $y=\min\{y_i,
y_j\}$, then $x-y\geqslant 2$,
$$
(x-1)^2+(y+1)^2=x^2-2x+1+y^2+2y+1=x^2+y^2-2(x-y-1)\leqslant
x^2+y^2-2.
$$
This means that if we apply consequent Shifts to $\lambda$, then
this process cannot be infinite.

\begin{lemma}
If $n\geqslant 3$ and $\lambda$ satisfies the conditions of
subsection~\ref{321}, then $M(\lambda)$ contains one of the points
$(2, 0, \dots, 0, -1, -1)$, $(1,1,0,\dots,0,-2)$, or $(1, 1,
0,\dots,0, -1, -1)$, and it always contains the point $(1, 0, \dots,
0, -1)$. \label{najdetsja}
\end{lemma}
\begin{proof}%[Proof of the lemma]
Let $\lambda=(a_1, \dots, a_n)$ (in the usual basis). If $\forall
i\; a_i\in\{-1,0,1\}$, then, due to the fact that
$\lambda\ne(1,0,\dots,0,-1)$, $\lambda$ has at least $4$ nonzero
coordinates. Taking into account that $\sum_1^n{a_i} = 0$, at least
two of them are equal to $1$ and two are equal to $-1$. In this
case, $M(\lambda)$ contains the point $(1,1,0,\ldots,0,-1,-1)$:
split all its other coordinates into pairs $1, -1$ and make them
zero (using the Shift), then permute the remaining $4$ coordinates.
Applying one more Shift, we yield $(1, 0, \dots, 0, -1)$.

Otherwise, if $\exists i, |a_i|>1$ (one of the coordinates is {\it
big}), then $\max_{i,j}(a_i-a_j)\geqslant 3$. Keeping at least one
coordinate big, perform the shift for the pairs of indices where
$|a_i-a_j|\geqslant 2$. This process is finite. Consider a situation
where we can perform no more Shift. If we still have a nonzero
coordinate with the same sign as the big coordinate has, we can
shift it with the coordinate of the opposite sign (their difference
will obviously be $\geqslant 2$). Otherwise we are in the case where
we have a big coordinate of one sign (without loss of generality
positive) and some coordinates of the opposite sign. If the big
coordinate is $\geqslant 3$, then apply a Shift to this coordinate
and to some negative coordinate. But we have supposed that Shifts
are impossible. Then the big coordinate is equal to $2$, nonzero
ones among the other coordinates are either $-2$ or $-1$ and $-1$.
But if $M(\lambda)$ contains a point $(2,0,0,\dots,0,-2)$, then it
also contains
$(2,0,\dots,0,-1,-1)=\frac12((2,0,\dots,0,-2)+(2,0,\dots,0,-2,0))$.
We can easily get $(1, 0, \dots, 0, -1)$, performing one more Shift.
\end{proof}

Construct NSSes for the first three points.

Example 8. $\lambda=(2, 0, \dots, 0, -1, -1),\; n\geqslant 3$.

Consider vectors
\begin{gather*}
v_1 = (1, -1, 0, 0,\dots,0),\\
v_2 = (-1, -1, 2, 0,\dots,0),\\
v_3 = (2, -1, -1, 0,\dots,0), \\
v = (0, -1, 1, 0,\dots,0) =\frac{1}{2} (v_1 + v_2) = v_2 + v_3 -v_1.
\end{gather*}
Suppose $f=-y_2$, then $f(v_1)=f(v_2)=f(v_3)=f(v)=1$, but $\forall
i\; v\ne v_i$. We get a contradiction.

The NSS for the point $\lambda=(1,1,0,\dots,0,-2),\; n\ge3$ can be
constructed similarly (one should multiply all the coordinates by
$-1$).

Example 9. $\lambda=(1, 1, 0,\dots,0, -1, -1)\in M(\lambda),\; n\ge
4.$

Consider vectors

\begin{gather*}
v_1 = (1, 1, -1, -1, 0,\dots,0),\\
v_2 = (1, -1, 1, -1, 0,\dots,0),\\
v_3 = (0, 1, 0, -1, 0,\dots,0),\\
v_4 = (0, 0, 1, -1, 0,\dots,0),\\
v = (1, 0, 0, -1, 0,\dots,0) = \frac{1}{2} (v_1 + v_2) = v_1 + v_4 -
v_3.
\end{gather*}

Suppose $f=-y_4$, then $f(v_1)=f(v_2)=f(v_3)=f(v_4)=f(v)=1$, but
$v\ne v_i$ for any $i$. We get a contradiction.

Now take an arbitrary dominant weight $\lambda$, $n\geqslant 3$, and
the corresponding set $M(\lambda)$. It follows from
Lemma~\ref{najdetsja} and the Inclusion Lemma that the NSS for
$\lambda$ exists.

\smallskip

It remains to consider the case $n=2$, $\lambda=(a,-a)$. If $|a|\ge
3$, then $\conv\{(\sigma(a_1, a_2)), \sigma \in S_2\}$ contains the
points $(2, -2)$, and $(3, -3)$. But this subset is not saturated:
$$
(1, -1)=\frac{1}{2}(2, -2)=(3, -3)-(2, -2),
$$
and the vector $(1, -1)$ is not a linear combination of vectors $(2,
-2 )$ and $(3, -3)$ with integer positive coefficients. If,
otherwise, $a \in \{0, \pm 1, \pm 2 \}$, then each subset in
$M(\lambda)$ is saturated (sections~\ref{ad} and \ref{4pi}).

\subsubsection{All the coordinates of $\lambda$ are non-integer.}

\begin{lemma}
Given a point $\lambda=(a_1, a_2, \dots, a_n)$ (in the usual basis),
$n\geqslant 4$. If the set $\{a_1, a_2, \dots, a_n\}$ contains
simultaneously $\alpha +1, \alpha, \alpha -1$ for some
$\alpha\in\RR$, then the set $M(\lambda)$ contains an
NSS.\label{aa-1a+1nenas}
\end{lemma}

\begin{proof}
It is easy to see that $M(\lambda)$ contains a point $v_1=(\alpha
+1, \alpha, \alpha -1, a_4, \dots, a_n)$, $a_4\ne 0$ (because
$a_4\not\in \mathbb Z$). Acting by $S_n$, we can get the following
points from it:
\begin{center}
\begin{tabular}{ccccccc}
$v_2$ & =&$($&$\alpha -1,$& $\alpha,$&$\alpha +1,$ & $a_4, \dots, a_n),$\\
$v_3$ &=& $($&$\alpha +1,$&$\alpha -1,$&$\alpha,$&$ a_4, \dots, a_n),$\\
$v_4$ &=& $($&$\alpha,$&$\alpha -1,$&$ \alpha +1,$&$ a_4, \dots,
a_n).$
\end{tabular}
\end{center}
Show that this set is not saturated. Indeed,
\begin{gather*}
\frac{1}{2}(v_1 + v_2) = (\alpha, \alpha, \alpha, a_4, \dots,
a_n),\\
v_2 + v_3 - v_4 = (\alpha, \alpha, \alpha, a_4, \dots, a_n),
\end{gather*}
suppose $f=\frac{y_4}{a_4}$, then
$$
f(v_1)=f(v_2)=f(v_3)=f(v_4)=f(v)=1.
$$
But $\forall i\quad v\ne v_i$, this means that $v$ is not a
$\ZZ_+$-combination of $v_i$.
\end{proof}

\begin{lemma}[Good Triple Lemma]
Let $\lambda=(a_1, \dots, a_n), n\geqslant 4$, and all $a_i$ are
non-integer. If the collection $a_1, \dots, a_n$ contains at least
three different values, then $M(\lambda)$ contains a point of form
$(\alpha +1, \alpha, \alpha -1, a_4, \dots, a_n)$.
\end{lemma}

\begin{proof}
Perform several Shifts preserving the condition that the set $\{a_1,
\dots, a_n\}$ contains at least 3 elements. Suppose further Shifts
are impossible (we mentioned above that, starting from any position,
only a finite number of Shifts is possible). Consider
$a_{\max}=\max\{a_1, \dots, a_n\}$, $a_{\min}=\min\{a_1, \dots,
a_n\}$, $a_{\midd}\in\{a_1, \dots, a_n\}$, $a_{\midd}\ne a_{\max}$,
$a_{\midd}\ne a_{\min}$. If $a_{\max}-a_{\midd}\geqslant 3$, then we
can apply the Shift to $a_{\max}$ and $a_{\midd}$, thus we obtain
three different values of coordinates $a_{\min}$, $a_{\midd}+1$,
$a_{\max}-1$. Similarly, if $a_{\midd}-a_{\min}\geqslant 3$, then at
least one more Shift is possible. So we yield $a_{\max}-a_{\midd}$,
$a_{\midd}-a_{\min}\in\{1,2\}$. If
$a_{\max}-a_{\midd}=a_{\midd}-a_{\min}=1$, we have already found a
point of necessary type in $M(\lambda)$. Up to symmetry, one of the
two cases is possible: either $a_{\min}=a_{\midd}-2$,
$a_{\max}=a_{\midd}+2$, or $a_{\min}=a_{\midd}-1$,
$a_{\max}=a_{\midd}+2$.  Consider these two cases.

In the first case, apply the Shift to $a_{\max}$ and $a_{\min}$.
This operation gives us the required triple
$(a_{\max}-1,a_{\midd},a_{\min}+1)$.

In the second case, $a_{\min}=a_{\midd}-1$, $a_{\max}=a_{\midd}+2$,
and we know that $\lambda$ has at least 4 coordinates. If there are
4 different values among them, the fourth will inevitably form a
triple of form $(\alpha+1,\alpha, \alpha-1)$ with two of $a_{\max}$,
$a_{\midd}$, $a_{\min}$. Otherwise $a_i\in
\{a_{\max},a_{\midd},a_{\min}\}$ for any $i$. But $n\geqslant 4$,
this means that at least one of the values $(a_1, a_2, \dots, a_n)$
is mentioned twice. Suppose $n=4$ (we need only $4$ $a_i$'s, forget
that there are other coordinates). The multiplicities of $(a_{\max},
a_{\midd}, a_{\min})$ may be as follows: $(\widehat{1,1,2})$,
$(\widehat{1,2},1)$, $(\widehat{2,1,1})$. Apply the shift to the
coordinates marked with the hat. We get one of the following
collections: $(a_{\midd}+1,a_{\midd},a_{\midd},a_{\midd}-1)$,
$(a_{\midd}+1,a_{\midd}+1,a_{\midd},a_{\midd}-1)$,
$(a_{\midd}+2,a_{\midd}+1,a_{\midd},a_{\midd})$. Each of them
contains a triple of form $(\alpha+1,\alpha, \alpha-1)$. But this
means that here we also find a triple of form $(\alpha+1,\alpha,
\alpha-1)$.
\end{proof}

\begin{lemma}[Absence of a good triple]
Let $\lambda=(a_1, \dots, a_n)$ be a dominant weight, $n\geqslant
4$, $\exists i$ with $|a_i|>1$, and all $a_i\not\in \mathbb Z$. If
$M(\lambda)$ does not contain a point of the form $(\alpha +1,
\alpha, \alpha -1, a_4, \dots, a_n)$, then
$$\lambda=\left(\frac{2n-2}{n}, -\frac{2}{n},
-\frac{2}{n}, \dots,
 -\frac{2}{n}\right)\; \text{or}\quad \lambda=\left(\frac{2}{n},
\frac{2}{n}, \dots,
 \frac{2}{n},-\frac{2n-2}{n} \right).
$$
\end{lemma}

\begin{proof}
If the collection $a_1, \dots, a_n$ contains at least $3$ different
elements, then we can use the Good Triple Lemma and show that
$M(\lambda)$ contains a point of the desired form. This means that
$\forall i\; a_i\in \{a_{\max}, a_{\min}\}$. Without loss of
generality, we may suppose that $a_{\max}>1$, and $a_{\min}<0$
(otherwise multiply all $a_i$ by $-1$).

If $a_{\min}<-1$, then apply the Shift to  $a_{\min}$ and
$a_{\max}$. Thus we get $a_{\min}+1$ and $a_{\max}-1$ among the
values of the coordinates, and still at least one of $a_{\min}$ and
$a_{\max}$ is presented (since $n\geqslant 4>3$). Using the Good
Triple Lemma, we get a contradiction.

We see that $-1<a_{\min}<0$. If the collection $(a_1, \dots, a_n)$
contains $a_{\max}$ at least for $2$ times, then apply the shift to
$a_{\min}$ and $a_{\max}$. Now we have $a_{\max}, a_{\min}+1>0$ and
at least one time $a_{\min}$ among the values of coordinates: all
the coordinates cannot be positive. This gives us a contradiction
with the Good Triple Lemma.

Then $a_{\max}$ enters only once in $(a_1, \dots, a_n)$. If
$a_{\max}>2$, apply the Shift to $a_{\min}$ and $a_{\max}$. We get
that
 $a_{\max}-1>1$, $a_{\min}+1<1$ and $a_{\min}$ are among the values of coordinates,
which gives us a contradiction with the Good Triple Lemma.

We yield that the collection has a form $(a_{\max}, a_{\min},
a_{\min}, \dots, a_{\min})$, $1<a_{\max}<2$, $-1<a_{\min}<0$. Let
$a_{\min} = -\frac kn$. We have $(n-1)a_{\min} + a_{\max} = 0$ from
the initial conditions. This yields $a_{\max} = \frac{k(n-1)}{n}$.
But $a_{\max}<2$. Consequently,
$$
\frac{k(n-1)}{n} < 2 \;\Rightarrow\; (n-1)<2n\;\Rightarrow\;
k<\frac{2n}{n-1}<3,
$$
because $n\geqslant 4$. Taking into account that $a_{\max} > 1$, we
get $k = 2$, $a_{\max}= \frac{2n-2}{n}$, $a_{\min}=-\frac 2n$. But
in the beginning of the case we could change the signs at all the
coordinates. Thus, we have two cases: $\lambda=\left(\frac{2n-2}{n},
-\frac{2}{n}, -\frac{2}{n}, \dots,
 -\frac{2}{n}\right)$ and $\lambda=\left(-\frac{2n-2}{n}, \frac{2}{n},
\frac{2}{n}, \dots, \frac{2}{n}\right) $.
\end{proof}

Applying the Lemmas, we see that in the case, when
$a_i\not\in\mathbb Z$, $\exists i, |a_i|>1$, $n\geqslant 4$, we have
not constructed an NSS only in these two cases. In all other cases
the NSS exists due to Lemma~\ref{aa-1a+1nenas}. Let us construct an
NSS in these two cases. We may assume that
$\lambda=\left(\frac{2n-2}{n}, -\frac{2}{n}, -\frac{2}{n}, \dots,
 -\frac{2}{n}\right)=2e_1$. Let
\begin{gather*}
 v_1 = 2e_1,\\
v_2 = 2e_2,\\
w = 2e_3,\\
v_3 = e_1+e_3 \in M(\lambda),\\
v_4 = e_2+e_3 \in M(\lambda).
\end{gather*}
Then $v_1$, $v_2$, $v_3$, $v_4$ form an NSS. Indeed, we have
\begin{gather*}
v = e_1+e_2=\frac12(v_1+v_2)= v_1 + v_4 - v_3,\\
f=x_1+x_2+x_3-3x_n.
\end{gather*}
Then $f(v_1)=f(v_2)=f(v_3)=f(v_4)=f(v)=2$, and $v\ne v_i$ for any
$i$. But $2$ cannot be represented as a sum of more than one $2$s.
We get a contradiction.

Now it remains to consider the cases $n = 3$ and $n=2$.

In the case $n=3$ we suppose that the fractional parts of all
 coordinates are equal to $\frac23$ (otherwise change
$\lambda$ for $-\lambda$, as we've done earlier). If $\lambda =
\left( \frac23, \frac23, -\frac43\right)$, then each subset in
$M(\lambda)$ is saturated (see~\ref{224}). Below we construct an NSS
for $\lambda=\left( \frac53, -\frac13, -\frac43\right)=3e_1+e_2$,
then, using the Inclusion Lemma, show the existence of NSS for all
other points $\lambda$. Let
\begin{gather*}
v_1 = e_1 =
\frac23(3e_1+e_2)+\frac13(e_2+3e_3),\\
v_2 = 2e_1+2e_2= \frac12(3e_1+e_2) +
\frac{1}{2}(e_1+3e_2),\\
v_3 = 3e_1+e_2,\\
v = 2e_1+e_2 = v_1 + \frac{1}{2} v_2 = v_3 - v_1.
\end{gather*}
If $f=x_1-x_3$, then $f(v_1)=1$, $f(v_2)=2$, $f(v_3)=3$, and
$f(v)=2$. But $v$ is equal neither to $v_2$, nor to $2v_1$. We get a
contradiction.

\begin{lemma}
Suppose that $\lambda=(a_1,a_2,a_3)$ (in the usual basis) is a
dominant weight such that the fractional parts of all $a_i$ are
equal to $\frac 23$. Suppose also that there exists an index $i$
with $|a_i|>1$, and $\lambda\ne\left( \frac{2}{3}, \frac{2}{3},
-\frac{4}{3}\right)$. Then $M(\lambda)$ contains a point $\left(
\frac{5}{3}, -\frac{4}{3}, -\frac{1}{3}\right)$.
\end{lemma}

\begin{proof}
It follows from the conditions of the lemma that $\exists i,
a_i\ge\frac 53$. Indeed, otherwise we have at least $2$ positive
coordinates, each of them $\le \frac 23$, but due to the condition
of the Lemma there exists an $a_i$ such that $|a_i|>1$. Suppose it
is $a_1$. We have $a_1=-a_2-a_3\geqslant -\frac 43$. This means that
$\lambda=\left( -\frac{4}{3}, \frac{2}{3}, \frac{2}{3}\right)$. We
get a contradiction.

If only one coordinate of $\lambda$ is positive, and it is equal to
$\frac 53$, then $\lambda=\left( \frac{5}{3}, -\frac{4}{3},
-\frac{1}{3}\right)$, and the Lemma is proved. Otherwise either
$\lambda$ has two positive coordinates, or one of them is $\geqslant
\frac 83$. In both cases we can apply the Shift to a positive and a
negative coordinate, such that after it $\lambda$ still has a
coordinate $\ge\frac 53$, and so on.
\end{proof}

Consider the case $n = 2$. Suppose $\lambda = (\frac{a}{2},
-\frac{a}{2})$, $a \in \mathbb N$, $a$ is odd. Then for $a=3$ each
subset in $M(\lambda)$ is saturated (see~\ref{3pi}). When $a
\geqslant 5$, we construct an NSS. Let
$$
v_1 = \left(\frac{3}{2}, -\frac{3}{2}\right),\; v_2 =
\left(\frac{5}{2}, -\frac{5}{2}\right).
$$
Then $\left(\frac{1}{2}, -\frac{1}{2}\right) = \frac{1}{3}v_1 = 2v_2
- v_1 \not\in \mathbb Z_+(v_1, v_2)$.

So we have constructed non-saturated subsets in the sets of weights
for all the representations not listed in the Main Theorem.

\end{document}